\documentclass[12pt]{amsart}
\usepackage{amssymb,amsmath,amsthm,mathtools,mathrsfs,hyperref}
\usepackage{enumerate}
\usepackage{graphicx}
\usepackage{array}
\usepackage{multicol}

\newtheorem{theorem}{Theorem}[section]
\newtheorem{lemma}[theorem]{Lemma}

\newtheorem{prop}[theorem]{Proposition}
\theoremstyle{definition}

\theoremstyle{remark}

\numberwithin{equation}{section}

\allowdisplaybreaks

\newcommand{\abs}[1]{\lvert#1\rvert}
\newcommand\mylabel[1]{\label{#1}}

\newcommand\thm[1]{\ref{thm:#1}}
\newcommand\lem[1]{\ref{lem:#1}}

\newcommand\propo[1]{\ref{propo:#1}}
\newcommand\eqn[1]{(\ref{eq:#1})}
\newcommand\sect[1]{\ref{sec:#1}}

\newcommand\restr[2]{{
  \left.\kern-\nulldelimiterspace 
  #1 
  \littletaller 
  \right|_{#2} 
  }}
\newcommand{\littletaller}{\mathchoice{\vphantom{\big|}}{}{}{}}  
\newcommand\mutwid{\widetilde{\mu}}
\DeclareMathOperator{\IM}{Im}
\DeclareMathOperator{\sgn}{sgn}

\begin{document}

\title[]
{Transformation of third order mock theta functions and new $q$-series identities}

\author{Frank Garvan}
\address{Department of Mathematics, University of Florida, Gainesville
FL 32611, USA}
\email{fgarvan@ufl.edu}
\author{Avi Mukhopadhyay}
\address{Department of Mathematics, University of Florida, Gainesville
FL 32611, USA}
\email{mukhopadhyay.avi@ufl.edu}

\date{\today}



\begin{abstract}
Ramanujan introduced mock theta functions in his last letter to G.H.Hardy. He provided examples and various relations between them. G.N.Watson found transformations for the third order mock theta functions $f(q)$ and $\omega$(q). Zwegers in 2000 built on Watson's techniques to complete these mock theta functions and connected them to real analytic modular forms. We show how to derive these transformations using Lerch sums. To show the equivalence of the results involves some new $q$-series identities thus resulting in a new proof of Zwegers' theorem.
\end{abstract}
\maketitle

\section{Introduction}
\mylabel{sec:intro}
Ramanujan in his last letter to G.H. Hardy in 1920 introduced what he called mock theta functions.
\begin{quote}
      I am extremely sorry for not writing you a single letter up to now . . . I discovered very interesting functions recently which I call 'Mock' $\theta$-functions. Unlike the 'False' $\theta$-functions (studied partially by Prof. Rogers in his interesting paper) they enter into mathematics as
beautifully as ordinary $\theta$-functions. I am sending you this letter with some examples . .
\end{quote}
Ramanujan wrote down a list of $17$ examples of mock $\theta$-functions and arranged them in terms of their 'order', a concept he did not define.However, he did mention properties which a mock $\theta$-function should satisfy.
Accordingly a mock $\vartheta$-function is a function $M$
of the complex variable $q$ (Ramanujan calls this the Eulerian
form), which converges for $|q|<1$ and satisfies the following
conditions\footnote{Stated in Rhoades\cite[p.7592]{Rhoades2013},
which follows the version of Ramanujan's definition given by Andrews
and Hickerson \cite{AndrewsHickerson1991} and Zwegers\cite{Zw2002}}:
\begin{enumerate}[(i)]
    \item $M(q)$ has infinitely many exponential singularities at roots of unity
    \item for every root of unity $\xi$, there is a $\vartheta$-function $\vartheta_{\xi}(q)$ such that the difference $M(q)-\vartheta_{\xi}(q)$ is bounded as $q$ approaches $\xi$ radially,
    \item There is no $\vartheta$-function that works for all $\xi$, i.e., $M(q)$ cannot be written as the sum of two functions, one of which is a $\vartheta$-function and the other which is bounded at all roots of unity.
\end{enumerate}
It seems that by a $\theta$-function Ramanujan means products and quotients of series of the form $\sum\limits_{n=-\infty}^{\infty}k^{n}q^{an^2+bn}$, where $k=-1,1$ and $a>0,b$ are rational. Some examples are the functions:
\begin{center}
\begin{enumerate}[(a)]
    \item \( \displaystyle f(q) = 1 + \frac{q}{(1+q)^2} + \frac{q^4}{(1+q)^2(1+q^2)^2} + \ldots \)
    
    \item \( \displaystyle \phi(q) = 1 + \frac{q}{(1+q^2)} + \frac{q^4}{(1+q^2)(1+q^4)} + \ldots \)
    
    \item \( \displaystyle \omega(q) = \frac{1}{(1-q)^2} + \frac{q^4}{(1-q)^2(1-q^3)^2} + \ldots \)
\end{enumerate}
\end{center}

The function $\omega(q)$ appears in Ramanujan's Lost Notebook and also in Watson's 1936 paper, The Final Problem\cite{Wa1936}, where alongside stating new examples he also proved some identities stated by Ramanujan in his last letter. He also formulated and proved modular transformation properties of third order mock theta functions.\\
From his work we see the following :

\begin{align*}
    q^{-1/24}f(q)=2\sqrt{\frac{2\pi}{\alpha}}q_{1}^{4/3}\omega(q_{1}^2)+4\sqrt{\frac{3\alpha}{2\pi}}\int\limits_{0}^{\infty} \dfrac{\sinh(\alpha t)}{\sinh(\frac{3\alpha t}{2})}e^{-\frac{3\alpha t^2}{2}}dt
\end{align*}
where $q:=\exp(-\alpha)$, $\beta:=\pi^{2}/\alpha$, $q_{1}:=\exp(-\beta)$ and $\alpha \in \mathbb{C}$ with Re$(\alpha)>0$.\\
In this identity we let $\alpha=-2 \pi i \tau $ so that  $q=\exp(2\pi i \tau)$, $q_{1}^{2}=\exp({\pi i \left(\frac{-1}{\tau}\right)})$ and we see that the term $2q_{1}^{4/3}\omega(q_{1}^2)$ occurs by applying $\tau \rightarrow \frac{-1}{\tau}$ to $2q^{\frac{1}{3}}\omega(q^{\frac{1}{2}})$ .

Basically Watson showed the following Lemma.
We define 
$F=(f_{0},f_{1},f_{2})^{T}$ where 
\[
f_0(\tau) = q^{-1/24} f(q), \quad f_1(\tau) = 2q^{1/3} \omega(q^{1/2}), 
\quad f_2(\tau) = 2q^{1/3} \omega(-q^{1/2}).
\]

\begin{lemma}[{ Watson \cite[pp.75-79]{Wa1936} }]
\mylabel{lem:watsontrans}
For  $q=e^{2 \pi i \tau}$, $\tau \in \mathbb{H}$  we have\\
\begin{align*}
F(\tau+1) &= \begin{pmatrix}\zeta_{24}^{-1} & 0 & 0\\0 & 0 &\zeta_{3}\\0 & \zeta_{3} & 0\end{pmatrix}F(\tau),\\
\text{and}\quad\dfrac{1}{\sqrt{-i\tau}}F(-1/\tau) &= \begin{pmatrix}0 & 1 & 0\\1 & 0 &0\\0 & 0 & -1\end{pmatrix}F(\tau)+R(\tau),     
\end{align*}
        with $\zeta_{n}=e^{2\pi i/n}, R(\tau)= 4\sqrt{3}\sqrt{-i\tau}(j_1(\tau)),-j_1(\tau)),j_3(\tau))^{T}$, where\\
           \begin{align*}
              j_{1}(\tau)=\int\limits_{0}^{\infty}e^{3\pi i \tau x^2}\dfrac{\sin(2\pi\tau x)}{\sin(3\pi \tau x)}dx,\\
              j_{2}(\tau)=\int\limits_{0}^{\infty}e^{3\pi i \tau x^2}\dfrac{\cos(\pi\tau x)}{\cos(3\pi \tau x)}dx,\\
              j_{3}(\tau)=\int\limits_{0}^{\infty}e^{3\pi i \tau x^2}\dfrac{\sin(\pi\tau x)}{\sin(3\pi \tau x)}dx.
\end{align*}
\end{lemma}

 Zwegers\cite{Zw2002} defined a new function which was the component-wise 
period integral of certain  theta functions of weight $3/2$.
  For $\tau \in \mathbb{H}
  \cup \mathbb{Q}$ he defined,
\begin{align*}
    G(\tau) := 2 i \sqrt{3}\int\limits_{-\overline{\tau}}^{ i \infty}\dfrac{(g_1(z),g_0(z),-g_2(z))^T}{\sqrt{- i (z+\tau)}}\,dz.
\end{align*}
     where
     \begin{align*}
g_0(z) &:= \sum\limits_{n\in\mathbb{Z}}(-1)^n(n+1/3)e^{3\pi i(n+1/3)^2z},\\
g_1(z) &:= -\sum\limits_{n\in\mathbb{Z}}(n+1/6)e^{3\pi i(n+1/6)^2z},\\
g_2(z) &:= \sum\limits_{n\in\mathbb{Z}}(n+1/3)e^{3\pi i(n+1/3)^2z}.
\end{align*}
In \cite[Lemma 3.3]{Zw2002} he proved that
\begin{align*}
    R(\tau)=-2i\sqrt{3}\int\limits_{0}^{i\infty}\frac{g(z)}{\sqrt{-i(z+\tau)}}dz
\end{align*}
where $g(z)=g_0(z),g_1(z),g_2(z))^T$ and used it to prove the following.
\begin{lemma}
For $\tau \in \mathbb{H}$, we have
\begin{align*}
    G(\tau + 1) &= 
    \begin{pmatrix}
        \zeta_{24}^{-1} & 0 & 0 \\
        0 & 0 & \zeta_{3} \\
        0 & \zeta_{3} & 0
    \end{pmatrix}
    G(\tau), \\
    \frac{1}{\sqrt{-i\tau}} G(-1/\tau) &=
    \begin{pmatrix}
        0 & 1 & 0 \\
        1 & 0 & 0 \\
        0 & 0 & -1
    \end{pmatrix}
    G(\tau) + R(\tau).
\end{align*}
\end{lemma}

This result together with \ref{lem:watsontrans} implies,

\begin{theorem}[{ Zwegers \cite[Theorem 3.6]{Zw2002} }]
\mylabel{thm:Zwthm}
The function $H$ defined by
\[
H(\tau) := F(\tau) - G(\tau) = (h_0(\tau), h_1(\tau), h_2(\tau))
\]
is a vector-valued real analytic modular form of weight $1/2$ satisfying
\begin{align}
H(\tau + 1) &= 
\begin{pmatrix} 
\zeta_{24}^{-1} & 0 & 0\\ 
0 & 0 & \zeta_3\\ 
0 & \zeta_3 & 0
\end{pmatrix} H(\tau), \mylabel{eq:T-transformation} \\
\frac{1}{\sqrt{-i\tau}} H\left(-\frac{1}{\tau}\right) &= 
\begin{pmatrix} 
0 & 1 & 0\\ 
1 & 0 & 0\\ 
0 & 0 & -1
\end{pmatrix} H(\tau). \mylabel{eq:S-transformation}
\end{align}

\end{theorem}

One can observe that  \eqn{T-transformation} is easy to prove component-wise, however proving \eqn{S-transformation} is more difficult. Zwegers proof of \eqn{S-transformation} depends crucially on Watson's transformation result Lemma\ref{lem:watsontrans}. In this paper we give a new proof of \eqn{S-transformation} that is independent of Watson's transformation formulae for $f(q)$ and $\omega(q)$ in Lemma~\ref{lem:watsontrans}.
Our proof of Theorem \thm{Zwthm} uses only properties of Appell-Lerch 
sums from Chapter 1 of Zwegers' thesis\cite{Zw-thesis}. In the process of the proof we discover the following new $q$-series identities:
\begin{align}   
2q^{2}\omega(-q^3) &= -\dfrac{2i}{\sqrt{3}} - \dfrac{2}{3}\dfrac{\eta(\tau)^2\eta(4\tau)^2}{\eta(2\tau)^2\eta(6\tau)} +\dfrac{4}{3}\frac{e^{\pi i/3}}{E(q^6)}\sum\limits_{n \in \mathbb{Z}}\frac{(-1)^{n}\zeta_{3}^{n}q^{n^2+n}}{1+q^{2n+1}},
\mylabel{eq:NEWOMEGA}
\\
    2q^{2} \omega(q^{3})
    &=- \frac{2i}{\sqrt{3}}+\frac{2}{3} \frac{\eta(2\tau)^4}{\eta(6\tau)\eta(\tau)^2} 
    + \frac{4}{3} \frac{e^{- \pi i/3}}{E(q^6)}\sum\limits_{n \in \mathbb{Z}}\frac{(-1)^{n}\zeta_{3}^{2n}q^{n^2+n}}{1-\zeta_{3}q^{2n+1}}
    ,\mylabel{eq:NEWOMEGA2}\\
  q^{\frac{-1}{8}}f(q^3)&=\frac{1}{3}\frac{\eta(\tau)^{4}}{\eta(3\tau)\eta(2\tau)^{2}}-\frac{4}{3}\frac{q^{-1/8}}{E(q^3)}\sum\limits_{n \in \mathbb{Z}}\frac{(-1)^{n}\zeta_{3}^{n}q^{\frac{n^2+n}{2}}}{1+q^{2n}},
\mylabel{eq:NEWF}
\end{align}    
 where $\zeta_{n}=\exp(2\pi i/n)$, $E(q)=\prod_{n=1}^{\infty}(1-q^n)$ and $\eta(\tau)=q^{1/24}E(q)$
In Section \sect{mainproof}  we show that the third component of 
\eqn{S-transformation} reduces to \eqn{NEWOMEGA} and then prove 
\eqn{NEWOMEGA} by calculating $3$-dissections of certain eta-quotients 
and Appell-Lerch sums.
In Section \sect{other} we show the first component of 
\eqn{S-transformation} reduces to \eqn{NEWOMEGA2} and prove 
\eqn{NEWOMEGA2} by replacing $q$ by $-q$ in \eqn{NEWOMEGA}. This also 
involves properties of Appell-Lerch sums. This will complete the proof of 
\eqn{S-transformation} since its second component follows easily 
from the first. As a corollary we deduce \eqn{NEWF}.

\section{Preliminary Results}
\mylabel{sec:prelim}
Throughout the paper we will use the following standard notation.
For a non-negative integer $L$ we define the conventional $q$-Pochammer 
symbol as
\begin{align*} 
    (a)_L = (a;q)_L &:= \prod_{k=0}^{L-1}(1-aq^k),\\
    (a)_{\infty} = (a;q)_{\infty} 
   &:= \lim_{L\rightarrow\infty}(a)_L\,\,\text{where}\,\,\lvert q\rvert<1.  
\end{align*}
Next we define
\begin{align*}
    E(q) &:= (q)_{\infty}.
\end{align*}
The Dedekind eta function is given by
\begin{align*}
 \eta(\tau) 
 &= e^{\frac{2 \pi i \tau }{24}}\prod_{n=1}^{\infty}(1-e^{2\pi i n \tau})=q^{\frac{1}{24}}E(q),
\end{align*}
where $\tau\in\mathbb{H}$ (the complex upper-half plane) and satisfies
\begin{equation}
 \eta\left(\frac{-1}{\tau}\right)=\sqrt{-i\tau}\eta(\tau)
\mylabel{eq:Etatrans}
\end{equation}
 
\subsection{Zwegers' Completion of Appell-Lerch Series}
\mylabel{subsec:appell}

We follow Chapter 1 from Zwegers' thesis \cite{ZP2002} and define some functions and state some of their properties. 
\begin{equation*}
\vartheta(z,\tau) := 
  \sum\limits_{n\in\frac{1}{2}+\mathbb{Z}}
   e^{\pi i n^2\tau+2\pi i n(z+\frac{1}{2})},
\end{equation*}
for $z\in\mathbb{C}$, and $\tau\in\mathbb{H}$. 
We have the Jacobi triple product identity \cite[p.8]{Zw-thesis}
\begin{equation*}
\vartheta(z;\tau)
=-iq^{\frac{1}{8}}\zeta^{-\frac{1}{2}}
   \prod_{n=1}^{\infty} (1-q^n)(1-\zeta q^{n-1})(1-\zeta^{-1}q^{n}), 
\end{equation*}
where $q= e^{2 \pi i \tau}$ and $\zeta= e^{2 \pi i z}$.\\
We also use another version of the Jacobi's triple product dentity as follows:\\
\begin{equation}
\left(q^2, qz, q/z; q^2\right)_\infty 
= \sum_{n=-\infty}^{\infty} (-1)^n z^n q^{n^2}. 
\mylabel{eq:JTP}
\end{equation}

For  $u,v\in\mathbb{C}\setminus\mathbb{Z}\tau+\mathbb{Z}$ and 
$\tau \in \mathbb{H}$ we define the Appell-Lerch series
\begin{equation*}
\mu(u,v;\tau) = \frac{e^{\pi i u}}{\vartheta(v;\tau)}
\sum\limits_{n\in\mathbb{Z}}
\dfrac{(-1)^ne^{\pi i(n^2+n)\tau+2\pi i nv}}{1-e^{2\pi i n\tau+2\pi i u}}.
\end{equation*}

Zwegers also defined:
\begin{equation*}
 R(u;\tau):=\sum\limits_{n \in 1/2+\mathbb{Z}}
    \{sgn(n)-E\left((n+a)\sqrt{2y}\right)\}
      (-1)^{n-1/2}e^{-\pi i n^2 \tau- 2\pi i nu},
\end{equation*}
where $u\in\mathbb{C}$, $\tau\in\mathbb{H}$, $y=Im(\tau)$, 
$a=\dfrac{Im(u)}{Im(\tau)}$ and for  $z \in \mathbb{C}$, $E(z)$ is defined by
\begin{equation*}
    E(z):=2\int\limits_{0}^{z}e^{-\pi u^2}du
         =\sum\limits_{n=0}^{\infty}\dfrac{(-\pi)^n z^{2n+1}}{n!(n+1/2)}.
\end{equation*}
For $z\in\mathbb{R}$
\begin{equation*}
    E(z)=sgn(z)(1-\beta(z^2)),
\end{equation*}
where
\begin{equation*}
    \beta(x)=\int\limits_{x}^{\infty}u^{-1/2}e^{-\pi u}du
\end{equation*}
 for $x\ge 0$.

The function $R$ has elliptic transformation properties.
\begin{prop}
\mylabel{propo:Rellprops}
We have
\begin{enumerate}
    \item[\normalfont(a)] 
    \mylabel{prop:Rellprops:a}
    $\displaystyle R(u+1)=-R(u)$,
    
    \item[\normalfont(b)] 
    \mylabel{prop:Rellprops:b}
    $\displaystyle R(u)+ e^{-2\pi iu -\pi i \tau}R(u+\tau)
     =2e^{-\pi i u-\pi i \tau/4}$,
    
    \item[\normalfont(c)] 
    \mylabel{prop:Rellprops:c}
    $\displaystyle R(-u)=R(u)$.
\end{enumerate}
\end{prop}

Zwegers was able to complete the Appell-Lerch function $\mu(u,v,\tau)$ using $R(u;\tau)$ so that the completed function has elliptic and modular transformation properties.
\begin{prop}
\mylabel{propo:mutwidprops}
We define
\[
    \mutwid(u,v;\tau) = \mu(u,v;\tau) + \frac{i}{2}R(u-v;\tau).
\]
Then
\begin{enumerate}[\normalfont(a)]
    \item $\mutwid(u+k\tau+l,\,v+m\tau+n)
    = (-1)^{k+l+m+n} e^{\pi i (k-m)^{2}\tau + 2\pi i (k-m)(u-v)}\mutwid(u,v)$\\ for $k,l,m,n \in \mathbb{Z}$,

    \item  \[
    \mutwid\!\left(\frac{u}{c\tau+d},\,\frac{v}{c\tau+d};\,\frac{a\tau+b}{c\tau+d}\right)
    = V(\gamma)^{-3} \sqrt{c\tau+d}\,
    e^{-\pi i c(u-v)^2/(c\tau+d)}\,\mutwid(u,v;\tau),
    \]
    for $\gamma =
    \begin{pmatrix}
        a & b \\ c & d
    \end{pmatrix}
    \in SL_{2}(\mathbb{Z})$
    and $V(\gamma)
    = \dfrac{\eta(\gamma \tau)}{\sqrt{c\tau+d}\,\eta(\tau)}$,

    \item $\mutwid(-u,-v) = \mutwid(v,u) = \mutwid(u,v)$.
\end{enumerate}
\end{prop}

Zwegers related the function $R(u;\tau)$ to the weight $3/2$ defined the 
unary theta functions:
\begin{align*}
 g_{a,b}(\tau):= \sum\limits_{n\in a+ \mathbb{Z}}ne^{\pi i n^2\tau+2\pi i n b}
\end{align*}
for $a,b \in \mathbb{R}, \tau \in \mathbb{H}$
\begin{prop}
\mylabel{propo:gabprops}
We have 
\begin{enumerate}[\normalfont(i)]
    \item $g_{a+1,b}(\tau)$=$g_{a,b}(\tau)$,
    \item $g_{a,b+1}(\tau)$=$e^{2\pi i a}g_{a,b}(\tau)$,
    \item $g_{-a,-b}(\tau)$=$-g_{a,b}(\tau)$,
    \item$ g_{a,b}(\tau+1)=e^{- \pi i a(a+1)}g_{a,a+b+\frac{1}{2}}$,
\item $ g_{a,b}(\dfrac{-1}{\tau})=i e^{2\pi i ab}(-i\tau)^{3/2}g_{b,-a}(\tau)$\mylabel{gabinverse}.

\end{enumerate}    
\end{prop}
\begin{theorem}
\mylabel{thm:gabints}
Let $\tau$ $\in$ $\mathbb{H}$. For $a \in (-1/2,1/2)$ and $b \in \mathbb{R}$.
Then
\begin{align*}
\int\limits_{-\overline{\tau}}^{i\infty}\dfrac{g_{a+1/2,b+1/2}(z)\,dz}{\sqrt{i(z+\tau)}} = -e^{- \pi i a^2 \tau + 2 \pi i a (b+1/2)}R(a\tau-b;\tau). 
\end{align*} 
\end{theorem}

We need to extend this theorem to the case $a=-1/2$. 
This extension was observed by Kang \cite[Lemma 2.1]{Ka2009}
and Jennings-Shaffer \cite[p.338]{JS2016}. We prove the result in detail.

\begin{lemma}\mylabel{lem:Rext}
   Let  $\tau \in \mathbb{H}$ and $b \in \mathbb{R}$. Then 
   \begin{equation*}
R(-\tau/2-b;\tau)
=e^{\pi i \tau/4 + \pi i b}
-e^{\pi i \tau/4 + \pi i(b+1/2)}
\int_{-\overline{\tau}}^{i\infty}
\dfrac{g_{0,b+1/2}(z)}
      {\sqrt{-i(z+\tau)}}\,dz.
\end{equation*} \mylabel{lem:lemmaG}
\end{lemma}
\begin{proof}
We proceed as in Zwegers' proof of Theorem \thm{gabints}. Here
$$
g_{1,b+1/2}(z) = g_{0,b+1/2}(z) = 
\sum\limits_{n \in \mathbb{Z}} ne^{\pi i n^2 z+ 2\pi i n(b+1/2)},
$$
where we can omit the term $n=0$.
This is uniformly bounded for $z$ in $\mathbb{H}$ away from zero, and for such $z$
the series is absolutely and uniformly convergent. We see that the
integral 
$$
\int\limits_{-\overline{\tau}}^{i\infty}\dfrac{g_{0,b+1/2}(z)}
            {\sqrt{-i(z+\tau)}} \,dz
$$ 
converges absolutely  and by uniform convergence we have
\begin{align*}
&\int\limits_{-\overline{\tau}}^{i\infty}\dfrac{g_{0,b+1/2}(z)dz}
            {\sqrt{-i(z+\tau)}}\,dz\\
&=\sum\limits_{n \in \mathbb{Z}, n \neq 0}
\int\limits_{-\overline{\tau}}^{i\infty}
\frac{ne^{\pi i n^2 z+ 2\pi i n(b+1/2)}}
         {\sqrt{-i(z+\tau)}}dz \qquad\mbox{($z=u - \tau$, dz = du)}\\
&=\sum\limits_{n \in \mathbb{Z}, n \neq 0}
  \int\limits_{\tau-\overline{\tau}}^{i\infty}
   \frac{ne^{\pi i n^2 (u-\tau)+ 2\pi i n(b+1/2)}}
        {\sqrt{-iu}}\,du           \qquad\mbox{($u = iv$, $du = i dv$)}  \\
&=i \sum\limits_{n \in \mathbb{Z}, n \neq 0}\int\limits_{2\IM(\tau)}^{\infty}
        \frac{ne^{\pi i n^2 (iv-\tau)+ 2\pi i n(b+1/2)}}
             {\sqrt{v}}\,dv  \\
&=i \sum\limits_{n \in \mathbb{Z}, n \neq 0}
    ne^{-\pi i n^2\tau+ 2\pi i n(b+1/2)}
\int\limits_{2\IM(\tau)}^{\infty} \frac{e^{-\pi n^2 v}}{\sqrt{v}} dv
   \qquad\mbox{($n^2 v = w$, $dv=dw/n^2$)}\\
&= i\sum\limits_{n \in\mathbb{Z}, n \ne0} \sgn(n) 
    e^{-\pi i n^2\tau+ 2\pi i n(b+1/2)}
\int\limits_{2n^2\IM(\tau)}^{\infty} \frac{e^{-\pi w}}{\sqrt{w}} dw\\
&=
i\sum\limits_{n \in \mathbb{Z}, n \neq 0}\sgn(n)
  e^{-\pi i n^2 \tau+ 2\pi i n(b+1/2)}\beta(2n^2\IM(\tau))\\
&=
i\sum_{n \in \mathbb{Z}, n \neq 0}(\sgn(n) - E(n\sqrt{2\IM(\tau)}))
  e^{-\pi i n^2 \tau+ 2\pi i n(b+1/2)}
 \qquad\mbox{(since $\beta(z^2) = 1 - \sgn(z) E(z)$)}\\
&= -i
+ i\sum_{n \in \mathbb{Z}}(\sgn(n) - E(n\sqrt{2\IM(\tau)}))
  e^{-\pi i n^2 \tau+ 2\pi i n(b+1/2)}  \\
&= -i
+ i\sum_{n \in 1/2+\mathbb{Z}}(\sgn(n-1/2) - E((n-1/2)\sqrt{2\IM(\tau)}))
  e^{-\pi i (n-1/2)^2 \tau+ 2\pi i (n-1/2)(b+1/2)}  \\
&= -i
+ ie^{-\pi i \tau/4 - \pi i b} 
   \sum_{n \in 1/2+\mathbb{Z}}(\sgn(n) - E((n-1/2)\sqrt{2\IM(\tau)}))
  (-1)^{n-1/2}
  e^{-\pi i n^2 \tau - 2\pi i n (-\tau/2-b))}  \\
&= -i
+ e^{-\pi i \tau/4 - \pi i (b-1/2) } 
   R(-\tau/2 -b ; \tau).
        \end{align*}
Thus we have shown 
$$                   
\int_{-\overline{\tau}}^{i\infty}
     \dfrac{g_{0,b+1/2}(z)}
           {\sqrt{-i(z+\tau)}} \,dz 
= -i - e^{-\pi i \tau/4 - \pi i (b + 1/2)}  R(-\tau/2-b;\tau),
$$
and the result follows.
\end{proof}

 \subsection{Some $q$-series identities }
\mylabel{subsec:someqids}
We will need the following $q$-series identities:
\begin{align}   
\dfrac{(q)_\infty}{(zq,z^{-1}q;q)_\infty} &= \dfrac{(1-z)}{(q)_\infty}\sum\limits_{n=-\infty}^{\infty}\dfrac{(-1)^n q^{n(n+1)/2}}{1-zq^n},
 \mylabel{eq:CRANK}\\
\sum\limits_{n = -\infty}^{\infty}(-1)^n q^{n^2} z^{n}
\dfrac{1 - z q^{2n}}{1 + z q^{2n}}
&= \dfrac{\Theta(z,q^2)\Theta(-zq,q^2)\Theta_{3}(q)}{\Theta(-z,q^2)},\mylabel{eq:thetaid}
\\
\omega(q) 
&= \dfrac{1}{(q^2;q^2)_{\infty}}\sum\limits_{n=-\infty}^{\infty}\dfrac{(-1)^n q^{3n(n+1)}}{1-q^{2n+1}},
\mylabel{eq:OMEGAWATSON}\\
f(q)&=\frac{1}{(q;q)_{\infty}}\sum\limits_{n=-\infty}^{\infty}\frac{(-1)^{n}q^{\frac{n(3n+1)}{2}}}{1+q^{n}},
\mylabel{eq:fidWAT}
\end{align}    
where 
$$
\Theta(z,q) = (z;q)_{\infty}(z^{-1}q;q)_{\infty}(q;q)_{\infty} \quad
\mbox{and}\quad \Theta_{3}(q) 
= \sum\limits_{n = -\infty}^{\infty} q^{n^2}.
$$
 Note that the left hand side of \eqn{CRANK} is the generating function 
for the vector crank of a partition \cite[p.168]{An-Ga1988} and the equality 
is established by. \cite[p.170]{Ekin}. The third and fourth identities
\eqn{OMEGAWATSON} and \eqn{fidWAT}  are due to Watson \cite[pp.64-66]{Wa1936}.

\subsubsection*{Proof of \eqn{OMEGAWATSON}}
    Let  $F(z)=F(z,q)$ be the left hand side of \eqn{thetaid}. Then
\begin{align*}
F(zq^2)
&= \sum_{n=-\infty}^{\infty} (-1)^n q^{n^2 + 2n}z^n 
   \frac{(1 - z q^{2n + 2})}{1 + z q^{2n + 2}} 
= -\sum_{n=-\infty}^{\infty} (-1)^n q^{n^2 - 1} z^{n-1}
   \frac{1 - z q^{2n}}{1 + z q^{2n}}
   \qquad (\text{replacing n by n-1}) \\[4pt]
&= -\frac{1}{zq} F(z).
\end{align*}
Now let $G(z)$ be the right hand side of \eqn{thetaid}. 
We observe that 
$$
G(zq^2)=\frac{-1}{zq}G(z).
$$ 
Now let $H(z)=F(z)-G(z)$, so that 
$$
H(zq^2)=\frac{-1}{zq}H(z).
$$
 By Lemma 2 in \cite[p.88]{At-SwD1954} , in the region 
$\abs{q^2}<\abs{z}\leq1$ 
the number of poles minus  number of zeroes of $H(z)$ , 
is either $-1$ or $H(z)\equiv 0$. 
We see that the only possible pole of $H(z)$ in the region 
$\abs{q^2}<\abs{z}\leq1$ could be at $z=-1$.
Near $z=-1$
$$                     
        F(z)=\frac{2}{1+z}+....,\quad\mbox{and}\quad 
G(z)=\frac{2}{1+z}\alpha(q)+...
$$              
where 
$$
\alpha(q)=\frac{(-q^2;q^2)_{\infty}(q;q^2)_{\infty}^2(q^2;q^2)_{\infty}\Theta_{3}(q)}
               {(q^2;q^2)^3} =1
$$            
since
$$                   
\frac{\theta(z,q^2)\Theta(-zq,q^2)\Theta_{3}(q)}{\Theta(-z,q^2)}
        =\frac{(z;q^2)(z^{-1}q^2;q^2)(-zq;q^2)(-z^{-1}q;q^2)\Theta_{3}(q)(q^2;q^2)}{(-z;q^2)(z^{-1}q^2;q^2)}.
$$                     
So, $H(z)$ has no pole at $z=-1$. 
It suffices to show that $H(z)=0$ for two values of $z$ in the 
region $\abs{q}^2<\abs{z}\leq 1$. 
We have
$$                  
F(1)=\sum\limits_{n = -\infty}^{\infty}(-1)^n q^{n^2}\dfrac{1-q^{2n}}{1+q^{2n}}
=\sum\limits_{n = -\infty}^{\infty}(-1)^n q^{n^2}\dfrac{1-q^{-2n}}{1+q^{-2n}}=-F(1)
$$               
thus $F(1)=0$.
Clearly $H(z)=0$ at $z=1$, since $G(1)=0$. Note $ G(-q) = 0$.

\begin{align*}
    F(-q)&=\sum\limits_{n = -\infty}^{\infty}q^{n^2+n}\dfrac{1+q^{2n+1}}{1-q^{2n+1}}=\sum\limits_{n = -\infty}^{\infty}q^{n^2+n}\dfrac{1+q^{-2n-1}}{1-q^{-2n-1}}=-F(-q)   \qquad (n \to -n-1).
\end{align*}
Hence $H(z)=0$ at $z=-q$ and thus the result follows.
\qed

In a private communication, George Andrews and Ole Warnaar have noted 
that 
\eqn{thetaid} can be proved from Bailey's ${}_{6}\psi_{6}$ summation formula~\cite[p.~239]{GasperRahman}:
In a private communication, George Andrews and Ole Warnaar have noted that 
\eqn{thetaid} can be proved from Bailey's ${}_{6}\psi_{6}$ summation 
formula~\cite[p.~239]{GasperRahman} (let
$a = z$,  
$b = i z^{1/2}$,  
$c = -i z^{1/2}$,  
$d = i \left( \frac{q}{t} \right)^{1/2}$,  
$e = -i \left( \frac{q}{t} \right)^{1/2}$,  and
$t \to 0$.)

\section{Proof of main results}
\mylabel{sec:mainproof}
In this section we give a detailed proof of the third component in 
\eqn{S-transformation} by proving the identity \eqn{NEWOMEGA} and 
in the next section we give a brief sketch of the other two transformations. 
First we write $h_{2}$ in terms of Zwegers' $\mutwid$ function and then look at 
its transformation. In the process we discover the identity \eqn{NEWOMEGA} 
and show that it's proof is equivalent to proving the third component 
of \eqn{S-transformation}. In the following Sections 3.3 and 3.4 
we find the $3-$dissection of an eta quotient, the $3-$dissection of a 
$\mu$ sum and prove relations. Using the $3$-dissections we prove \eqn{NEWOMEGA} 
in Section 3.5 thus proving the third component of \eqn{S-transformation}.

\subsection{$h_{2}(\tau)$ in terms of the $\mutwid$ function}
\mylabel{subsec:h2andmutilde}
We rewrite Zwegers' 
$h_{2}(\tau)$ in terms of his ecompleted $\mutwid$-function.  
\begin{theorem}
\mylabel{thm:h2thm}
We have
\begin{align*}
h_{2}(\tau) = 
2\dfrac{\eta(6\tau)^2\eta(3\tau/2)^2}{\eta(3\tau)^2\eta(\tau)} 
- 4q^{-1/24}\mu\left(3\tau/2 + 1/2,\tau;3\tau\right).
\end{align*}
\end{theorem}

\begin{proof}
Recall that,
$$                   
h_2(\tau) = 2q^{1/3}\omega(-q^{1/2}) 
- 2i\sqrt{3}\int\limits_{-\overline{\tau}}^{i\infty} 
\dfrac{g_{2}(z)\,dz}{\sqrt{i(z+\tau)}},
$$              
where
$$                  
g_{2}(z) = \sum\limits_{n\in\mathbb{Z}}(n+1/3) e^{3\pi i (n+1/3)^{2}z}.
$$                
We observe:
\begin{align*}
\int\limits_{-\overline{\tau}}^{i\infty}
\dfrac{g_{2}(z)\,dz}{\sqrt{i(z+\tau)}} 
&= \int\limits_{-\overline{\tau}}^{i\infty}
\dfrac{g_{1/3,0}(3z)\,dz}{\sqrt{i(z+\tau)}} 
= \dfrac{1}{\sqrt{3}} \int\limits_{-3\overline{\tau}}^{i\infty}
\dfrac{g_{1/3,0}(z)\,dz}{\sqrt{i(z+3\tau)}} \\
&= \dfrac{-1}{\sqrt{3}}q^{-1/24}
R\left(-\tau/2 + 1/2;\,3\tau\right),
\end{align*}
by using $a=-1/6$ and $b=-1/2$ in Theorem \thm{gabints}.
\end{proof}
Now, from \eqn{OMEGAWATSON} we have
\begin{align*}
&\omega(-q^{1/2}) = \dfrac{1}{(q;q)_{\infty}}\sum\limits_{n=-\infty}^{\infty}\dfrac{(-1)^n q^{3n(n+1)/2}}{1+q^{n+1/2}}
=\dfrac{1}{(q;q)_{\infty}}\sum\limits_{n= -\infty}^{\infty}\dfrac{(-1)^n q^{3n(n+1)/2}(1-q^{n+1/2}+q^{2n+1})}{1+q^{3n+3/2}}
\\
&=\dfrac{1}{(q;q)_{\infty}}\sum\limits_{n=-\infty}^{\infty}\dfrac{(-1)^n q^{3n(n+1)/2}}{1+q^{3n+3/2}}
-\dfrac{q^{1/2}}{(q;q)_{\infty}}\sum\limits_{n= -\infty}^{\infty}\dfrac{(-1)^n q^{3n(n+1)/2+n}}{1+q^{3n+3/2}}
+\dfrac{q}{(q;q)_{\infty}}\sum\limits_{n=-\infty}^{\infty}\dfrac{(-1)^n q^{3n(n+1)/2+2n}}{1+q^{3n+3/2}}\\
&=\dfrac{1}{(q;q)_{\infty}}\sum\limits_{n=-\infty}^{\infty}\dfrac{(-1)^n q^{3n(n+1)/2}}{1+q^{3n+3/2}}
-\dfrac{2q^{-1}}{(q;q)_{\infty}}\sum\limits_{n=-\infty}^{\infty}\dfrac{(-1)^{n}q^{3n(n+1)/2-n}}{1+q^{3n-3/2}} \\
&\qquad\mbox{(by changing $n$ to $n-1$ in the second sum and $n$ to $-n$ in the third sum)}\\
&= \dfrac{(q^6;q^6)_{\infty}^{2}(q^{3/2};q^{3/2})_{\infty}^{2}}{(q^3;q^3)_{\infty}^{2}(q;q)_{\infty}}
-\dfrac{2q^{-1}}{(q;q)_{\infty}}\sum\limits_{n=-\infty}^{\infty}\dfrac{(-1)^{n}q^{3n(n+1)/2-n}}{1+q^{3n-3/2}}
 \qquad\mbox{(replacing $q$ by $q^{3}$ and $z$ by $q^{3/2}$ in \eqn{CRANK})}   
\end{align*}

Hence,
\begin{align*}
2q^{1/3}\omega(-q^{1/2})=2\dfrac{\eta(6\tau)^2\eta(3\tau/2)^2}{\eta(3\tau)^2\eta(\tau)}-4q^{-1/24}\mu\!\left(\tfrac{-3\tau}{2}+\tfrac{1}{2},-\tau;3\tau\right)
\end{align*}

and
\begin{align}
h_2(\tau) &= 2\dfrac{\eta(6\tau)^2\eta(3\tau/2)^2}{\eta(3\tau)^2\eta(\tau)}
- 4q^{-1/24}\mu\!\left( -\tfrac{3\tau}{2} + \tfrac{1}{2}, -\tau; 3\tau \right)
- 2i q^{-1/24}R\!\left( -\tfrac{\tau}{2} + \tfrac{1}{2}; 3\tau \right) \\
&= 2\dfrac{\eta(6\tau)^2\eta(3\tau/2)^2}{\eta(3\tau)^2\eta(\tau)} 
- 4 q^{-1/24}\mutwid\!\left( \tfrac{-3\tau}{2} + \tfrac{1}{2}, -\tau; 3\tau \right) 
\mylabel{eq:h2mutilde}
\end{align}

\subsection{Transforming $h_2$ and a new identity for $\omega(q)$}
\mylabel{subsec:transh2andomega}
In this section, we show that the transformation
\begin{equation}
-\dfrac{1}{\sqrt{-i \tau}}h_2\left(\dfrac{-1}{\tau}\right) = h_{2}(\tau) 
\mylabel{eq:3rdocomptrans}
\end{equation}
is equivalent to \eqn{NEWOMEGA}.
By \eqn{h2mutilde}, \eqn{Etatrans} and Proposition \propo{mutwidprops}(b) 
we have
\begin{align*}
&h_{2}\left(\dfrac{-1}{\tau}\right) = 2\dfrac{\eta(-6/\tau)^2 \eta(-3/2\tau)^2}{\eta(-3/\tau)^2 \eta(-1/\tau)} - 4 \exp(\pi i / 12\tau) \mutwid(3/2\tau + 1/2, 1/\tau; -3/\tau)\\
&= \dfrac{2\sqrt{-i \tau} \eta(\tau/6)^2 \eta(2\tau/3)^2}{3 \eta(\tau/3)^2 \eta(\tau)} + 4 \exp\left(\dfrac{\pi i}{12\tau}\right) \sqrt{\dfrac{-i\tau}{3}} \exp\left(\dfrac{-\pi i}{12}(\tau + 1/\tau + 2)\right) \mutwid(\tau/6 + 1/2, 1/3; \tau/3).
\end{align*}
Therefore
\begin{align*}
\dfrac{-1}{\sqrt{-i\tau}} h_{2}\left(\dfrac{-1}{\tau}\right) = -\dfrac{2 \eta(\tau/6)^2 \eta(2\tau/3)^2}{3 \eta(\tau/3)^2 \eta(\tau)} - \dfrac{4}{\sqrt{3}} \exp\left(\dfrac{\pi i}{12\tau}\right) \exp\left(\dfrac{-\pi i}{12}(\tau + 1/\tau + 2)\right) \mutwid(\tau/6 + 1/2, 1/3; \tau/3),
\end{align*}

\begin{align*}
= -\dfrac{2 \eta(\tau/6)^2 \eta(2\tau/3)^2}{3 \eta(\tau/3)^2 \eta(\tau)} - \dfrac{4}{\sqrt{3}} \exp\left(\dfrac{-\pi i \tau}{12} - \dfrac{-\pi i}{6}\right) \mu(\tau/6 + 1/2, 1/3; \tau/3) \\
- \dfrac{2i}{\sqrt{3}} \exp\left(\dfrac{-\pi i \tau}{12} - \dfrac{-\pi i}{6}\right) R(\tau/6 + 1/6; \tau/3).
\end{align*}
By Proposition \propo{mutwidprops} and Lemma \lem{lemmaG} with   $b = 1/6$ and $\tau \rightarrow \tau/3$, we have
\begin{align*}
-\dfrac{1}{\sqrt{-i \tau}} h_{2}\left(-\dfrac{1}{\tau}\right) = 
-\dfrac{2i}{\sqrt{3}} - \dfrac{2}{3} \dfrac{\eta(\tau/6)^2 \eta(2\tau/3)^2}{\eta(\tau/3)^2 \eta(\tau)} 
&- \dfrac{4}{\sqrt{3}} \exp\left(-\dfrac{\pi i \tau}{12} - \dfrac{\pi}{6}\right) \mu(\tau/6 + 1/2, 1/3; \tau/3) \\
&- \dfrac{2}{3} \int\limits_{-\overline{\tau}}^{i \infty} \dfrac{g_{0,2/3}(z/3)\,dz}{\sqrt{-i(z+\tau)}}.
\end{align*}

We show that $g_{0,2/3}(z/3) = -i3\sqrt{3}g_{2}(z)$. 
Since, $g_{2}(-1/\tau)=(-i \tau)^{3/2}g_{2}(\tau)$ and $g_{2}(\tau)=g_{1/3,0}(\tau)$,
\begin{align*}
 g_{2}(\tau)=\dfrac{g_{1/3,0}\left(-\dfrac{1}{\tau/3}\right)}{(-i \tau)^{3/2}} 
=\dfrac{i(-i\tau/3)^{3/2} \, g_{0,-1/3}(\tau/3)}{(-i\tau)^{3/2}} 
= \dfrac{i \, g_{0,2/3}(\tau/3)}{3\sqrt{3}} 
\qquad \mbox{(by Proposition \propo{gabprops}(v))}
\end{align*}
Therefore, 
\begin{align*}
-\dfrac{1}{\sqrt{-i \tau}} h_{2}\left(-\dfrac{1}{\tau}\right) = 
-\dfrac{2i}{\sqrt{3}} - \dfrac{2}{3} \dfrac{\eta(\tau/6)^2 \eta(2\tau/3)^2}{\eta(\tau/3)^2 \eta(\tau)} 
&- \dfrac{4}{\sqrt{3}} \exp\left(-\dfrac{\pi i \tau}{12} - \dfrac{\pi}{6}\right) \mu(\tau/6 + 1/2, 1/3; \tau/3) \\
&- 2i\sqrt{3} \int\limits_{-\overline{\tau}}^{i \infty} \dfrac{g_{2}(z)\,dz}{\sqrt{-i(z+\tau)}}.
\end{align*}
Since
\begin{align*}
h_{2}(\tau) 
= 2q^{1/3} \omega(-\sqrt{q}) + 2i \sqrt{3} \int\limits_{-\overline{\tau}}^{i \infty} \dfrac{g_{2}(z)\,dz}{\sqrt{i(z+\tau)}}. 
\end{align*}
We see that the transformation 
\begin{equation*}
-\dfrac{1}{\sqrt{-i \tau}}h_2\left(\dfrac{-1}{\tau}\right) = h_{2}(\tau) 
\end{equation*} 
is equivalent to
\begin{align*}
    2q^{1/3} \omega(-\sqrt{q}) 
= -\dfrac{2i}{\sqrt{3}} - \dfrac{2}{3} \dfrac{\eta(\tau/6)^2 \eta(2\tau/3)^2}{\eta(\tau/3)^2 \eta(\tau)} 
- \dfrac{4}{\sqrt{3}} \exp\left(-\dfrac{\pi i \tau}{12} - \dfrac{\pi i}{6}\right) \mu(\tau/6 + 1/2, 1/3; \tau/3),
\end{align*}
which is equivalent to \eqn{NEWOMEGA} after replacing $\tau$ by $6\tau$.

\subsection{Dissection of an eta-quotient}
\mylabel{subsec:disseta}
\begin{lemma}\mylabel{lem:eta3diss}
$3$-dissection of the $\eta$-quotient:  
\[
\dfrac{E(q)^2 E(q^4)^2}{E(q^2)^2 E(q^6)} = e_0(q^3) - 2q e_1(q^3) + q^2 e_2(q^3),
\]
where
\[
e_0(q) = \dfrac{E(q^6)^{10} E(q^4)^2 E(q)^2}{E(q^{12})^4 E(q^3)^4 E(q^2)^5}, \quad
e_1(q) = \dfrac{E(q^6)^4 E(q^4) E(q)}{E(q^{12}) E(q^3) E(q^2)^3}, \quad
e_2(q) = \dfrac{E(q^{12})^2 E(q^3)^2}{E(q^6)^2 E(q^2)}.
\]

\begin{proof}
Let \begin{align*}
    \Delta(q) &= \sum\limits_{n \geq 0} q^{n(n+1)/2}= \sum\limits_{n = -\infty}^{\infty} q^{n(2n+1)}=\frac{E(q^2)^2}{E(q)}
    \text{  and}\\
    \Delta(-q) &= \sum\limits_{n \geq 0} (-q)^{n(n+1)/2}=\dfrac{E(q) E(q^4)}{E(q^2)}
\end{align*}
Then
\begin{align*}
     \dfrac{E(q)^{2} E(q^4)^{2}}{E(q^2)^{2}}= \left(\Delta(-q)\right)^{2}.
\end{align*}
We have
\begin{align*}
n(2n+1) &\equiv 0 \bmod{3} \quad \text{when } n \equiv 0, 1 \bmod{3}, \\
     n(2n+1) &\equiv 1 \bmod{3} \quad \text{when } n \equiv 2 \bmod{3}.
\end{align*}

So,
\begin{align*}
\Delta(q) &= P_0(q^3) + q P_1(q^3), \\
\text{where }
P_0(q) &= \sum\limits_{n = -\infty}^{\infty} q^{6n^2 + n} 
       + \sum\limits_{n = -\infty}^{\infty} q^{6n^2 + 5n} 
       = \sum\limits_{n = -\infty}^{\infty} q^{\frac{3n^2 + 3n}{2}} 
       = \dfrac{E(q^2) E(q^3)^2}{E(q^6) E(q)}, \\
\text{and }P_1(q) &= \sum\limits_{n = -\infty}^{\infty} q^{6n^2 - 3n} 
       = \dfrac{E(q^6)^2}{E(q^3)},
\end{align*}

from Jacobi's triple product identity \eqn{JTP}. We find

\[
\dfrac{\Delta^2(-q)}{E(q^6)} = \dfrac{(P_0(-q^3) - q P_1(-q^3))^2}{E(q^6)} = e_0(q^3) - 2q e_1(q^3) + q^2 e_2(q^3),
\]
where
\[
e_0(q) = \dfrac{E(q^6)^{10} E(q^4)^2 E(q)^2}{E(q^{12})^4 E(q^3)^4 E(q^2)^5}, \quad
e_1(q) = \dfrac{E(q^6)^4 E(q^4) E(q)}{E(q^{12}) E(q^3) E(q^2)^3}, \quad
e_2(q) = \dfrac{E(q^{12})^2 E(q^3)^2}{E(q^6)^2 E(q^2)}.
\]

\end{proof}
\end{lemma}
By another routine application of Jacobi's triple product identity \eqn{JTP}, we find that,
\begin{equation}
    \vartheta\left(\frac{1}{3}, 2\tau\right) = e^{\frac{5\pi i}{6}}\left(\frac{3}{2} + i \frac{\sqrt{3}}{2}\right) q^{1/4} E(q^6).
\end{equation}

\subsection{Dissection of a sum}
\mylabel{subsec:disssum}
We find the $3-$dissection of the sum in 
\begin{align*}
    \mu(\tau+1/2,1/3,2\tau)=\frac{e^{\pi i (\tau+1/2)}}{\vartheta(1/3;2\tau)}\sum\limits_{n=-\infty}^{\infty}\dfrac{(-1)^n e^{2\pi i n/3}q^{n^2+n}}{1+q^{2n+1}}.
    \end{align*}
\begin{lemma}
\mylabel{lem:mudiss}
We define the functions $Y_j$ and $Y_{jk}$ by
\begin{align*}
    \sum\limits_{n=-\infty}^{\infty}\dfrac{(-1)^ne^{2\pi i n/3}q^{n^2+n}}{1+q^{2n+1}} = Y_0 + \zeta Y_1 +\zeta^2 Y_2,
\end{align*}
where $\zeta = e^{2\pi i/3}$ and
\begin{align*}
    Y_{j} &= \sum\limits_{\substack{n =- \infty\\n\equiv j\bmod 3}}^{\infty}\dfrac{(-1)^n q^{n^2+n}}{1+q^{2n+1}}=\sum\limits_{k=0}^{2}q^{k}Y_{jk}(q^3),
\end{align*}
for $0\leq j\leq 2$\\

Then
\begin{multicols}{2}
 \begin{enumerate}[(i)]
        \item $Y_0-Y_2 = E(q^6)$,\\
        \item  $\dfrac{Y_{00}(q)}{E(q^2)} = 1/2(e_0(q)+1)$,\\
        \item $\dfrac{Y_{01}(q)}{E(q^2)} = -e_1(q)$,\\
        \item $Y_{10}(q) = Y_{11}(q) = 0$,\\
        \item $\dfrac{Y_{12}(q)}{E(q^2)} = -\omega(-q)$,\\
        \item  $2\dfrac{Y_{02}(q)}{E(q^2)} = \omega(-q)+e_2(q)$,
    \end{enumerate}
\end{multicols}
where $e_{0}(q)$, $e_{1}(q)$, $e_{2}(q)$ are given in Lemma \lem{eta3diss}. 
\end{lemma}
\begin{proof}
(i)                 
\begin{align*}
    Y_{0}-Y_{2}&=\sum\limits_{n \in \mathbb{Z}}\dfrac{(-1)^{n}q^{9n^2+3n}}{1+q^{6n+1}}+\sum\limits_{n \in \mathbb{Z}}\dfrac{(-1)^{n}q^{9n^2+15n+6}}{1+q^{6n+5}}
    =\sum\limits_{n \in \mathbb{Z}}\dfrac{(-1)^{n}q^{9n^2+3n}}{1+q^{6n+1}}+\sum\limits_{n \in \mathbb{Z}}\dfrac{(-1)^{n}q^{9n^2+9n+1}}{1+q^{6n+1}}\\
    &=\sum\limits_{n \in \mathbb{Z}}\dfrac{(-1)^{n}q^{9n^2+3n}(1+q^{6n+1})}{1+q^{6n+1}}=\sum\limits_{n \in \mathbb{Z}}(-1)^{n}q^{9n^2+3n} =E(q^6),
\end{align*}
by the Jacobi's triple product identity\eqn{JTP}.

(ii)                     
By Euler's pentagonal number theorem and \eqn{thetaid}
(with $q\rightarrow q^3$ and $z\rightarrow q$) we have
\begin{align*}
    2Y_{00}-E(q^2)&=2\sum\limits_{n \in \mathbb{Z}}\dfrac{(-1)^n q^{3n^2+n}}{1+q^{6n+1}}-\sum\limits_{n \in \mathbb{Z}}(-1)^nq^{3n^2+n}
    =\sum\limits_{n \in \mathbb{Z}}(-1)^n q^{3n^2+n}\left(\dfrac{2}{1+q^{6n+1}}-1\right)\\
    &=\sum\limits_{n \in \mathbb{Z}}\dfrac{(-1)^n q^{3n^2+n}(1-q^{6n+1})}{1+q^{6n+1}}= \dfrac{\theta(q,q^6)\theta(-q^4,q^6)\theta_{3}(q^3)}{\theta(-q,q^6)}=e_{0}(q)E(q^2) 
\end{align*}
  after some simplification using Jacobi's triple product identity. This completes the proof of (ii).

(iii) 
Using \eqn{CRANK}(with $q \rightarrow q^3, z=-q$)
\begin{align*}
    Y_{01}(q)&=\sum\limits_{n \in \mathbb{Z}}\dfrac{(-1)^nq^{3n^2+3n}}{1+q^{6n+1}}=\dfrac{E(q^6)^2}{(1+q)(-q^7;q^6)_{\infty}(-q^5;q^6)_{\infty}}
    =\dfrac{E(q^6)^4E(q^4)E(q)}{E(q^2)^2E(q^3)E(q^{12})}= E(q^2)e_{1}(q)
\end{align*}
thus proving (iii).

(iv)
Since
\begin{center}
    $Y_{1}(q) =- \sum\limits_{n \in \mathbb{Z}} \dfrac{(-1)^n q^{9n^2 + 9n + 2}}{1 + q^{6n + 3}}$
\end{center}

\medskip

It is clear that
\[
Y_{10}(q) = Y_{11}(q) = 0.
\]

\medskip

(v) We have
\begin{align*}
    Y_{12}(q) 
    = -\sum\limits_{n \in \mathbb{Z}} \dfrac{(-1)^n q^{3n^2 + 3n}}{1 + q^{2n + 1}} 
    = -\omega(-q) E(q^2)
\end{align*}
by \eqn{OMEGAWATSON}.

(vi)\[
\begin{aligned}
2Y_{02} - E(q^2)\omega(-q) 
&= 2 \sum\limits_{n \in \mathbb{Z}} \dfrac{(-1)^n q^{3n^2 + 5n}}{1 + q^{6n + 1}} 
  - \sum\limits_{n \in \mathbb{Z}} \dfrac{(-1)^n q^{3n^2 + 3n}}{1 + q^{2n + 1}} \\
&= 2 \sum\limits_{n \in \mathbb{Z}} \dfrac{(-1)^n q^{3n^2 + 5n}}{1 + q^{6n + 1}} 
  - \sum\limits_{n \in \mathbb{Z}} \dfrac{(-1)^n q^{3n^2 + 3n}(1 - q^{2n + 1} + q^{4n + 2})}{1 + q^{6n + 3}} \\
&= 2 \sum\limits_{n \in \mathbb{Z}} \dfrac{(-1)^n q^{3n^2 + 5n}}{1 + q^{6n + 1}} 
  - \sum\limits_{n \in \mathbb{Z}} \dfrac{(-1)^n q^{3n^2 + 3n}}{1 + q^{6n + 3}} \\
&\quad + \sum\limits_{n \in \mathbb{Z}} \dfrac{(-1)^n q^{3n^2 + 5n + 1}}{1 + q^{6n + 3}} 
  - \sum\limits_{n \in \mathbb{Z}} \dfrac{(-1)^n q^{3n^2 + 7n + 2}}{1 + q^{6n + 3}} \\
&= 2 \sum\limits_{n \in \mathbb{Z}} \dfrac{(-1)^n q^{3n^2 + 5n}}{1 + q^{6n + 1}} 
  - 2 \sum\limits_{n \in \mathbb{Z}} \dfrac{(-1)^n q^{3n^2 + 7n + 2}}{1 + q^{6n + 3}} 
  - \sum\limits_{n \in \mathbb{Z}} \dfrac{(-1)^n q^{3n^2 + 3n}}{1 + q^{6n + 3}} \\
&= 2 \sum\limits_{n \in \mathbb{Z}} \dfrac{(-1)^n q^{3n^2 + 5n}}{1 + q^{6n + 1}} 
  - 2q^2 \sum\limits_{n \in \mathbb{Z}} \dfrac{(-1)^n q^{3n^2 + 7n }}{1 + q^{6n + 3}} - E(q^2)e_2(q) \\
&= 2E(q^2)e_2(q) - E(q^2)e_2(q) = E(q^2)e_2(q)
\end{aligned}
\]

\noindent
by using \eqn{CRANK} with \( q \rightarrow q^6 \), \( z \rightarrow -q^3 \) and using Proposition 1.4(7) in [\cite{Zw2002}, p.9].
This completes the proof of (vi).
\end{proof}

\subsection{Bringing everything together}
\mylabel{subsec:toget}
    We complete the proof of our new omega identity \eqn{NEWOMEGA} 
$$                 
2q^{2}\omega(-q^3) = -\dfrac{2i}{\sqrt{3}} - 
      \dfrac{2}{3}\dfrac{E(q)^2  E(q^4)^2}{E(q^2)^2 E(q^6)} 
      - \dfrac{4}{\sqrt{3}}\exp\left({\dfrac{-\pi i \tau}{2}
      -\dfrac{\pi i}{6}}\right)\mu(\tau+1/2,1/3;2\tau).
$$                        
\begin{proof}
By Lemma \lem{mudiss} we have
\begin{align*}
Y_0(q^3) + \zeta Y_1(q^3) + \zeta^2 Y_2(q^3)
&= \frac{E(q^{6})}{2}\Big[
   (1+\zeta^{2})e_{0}(q^3) + (1-\zeta^{2}) \\
&\quad
   - 2q(1+\zeta^{2})e_{1}(q^3)
   + q^2(1+\zeta^{2})e_{2}(q^3)
   - 3q^2\zeta\,\omega(-q^3)
\Big].
\end{align*}

Therefore,
by Lemmas  \lem{eta3diss} and \lem{mudiss} we have

\begin{align*}
       & -\dfrac{2i}{\sqrt{3}} - \dfrac{2}{3}\dfrac{E(q)^2  E(q^4)^2}{E(q^2)^2 E(q^6)} - \dfrac{4}{\sqrt{3}} \exp\left({-\dfrac{\pi i\tau}{2}-\dfrac{\pi i}{6}}\right)\mu(\tau+1/2,1/3;2\tau)\\
     &=-\dfrac{2i}{\sqrt{3}} - \dfrac{2}{3}(\text{e}_0(q^3)-2q\text{e}_1(q^3)+q^2\text{e}_3(q^3))+\dfrac{4}{3}\frac{e^{\pi i/3}}{E(q^6)}\left( Y_0(q^3) + \zeta Y_1(q^3) +\zeta^2 Y_2(q^3)\right)\\
      &=-\dfrac{2i}{\sqrt{3}} - \dfrac{2}{3}(\text{e}_0(q^3)-2q\text{e}_1(q^3)+q^2\text{e}_3(q^3))+\dfrac{2}{3}\left( 1+2\zeta+ e_{0}(q^3)-2qe_{1}(q^3)+q^2e_{2}(q^3)+3q^2\omega(-q^3)\right)\\
     &=2q^{2}\omega(-q^3)
\end{align*}
after some simplification.
This proves the new omega identity \eqn{NEWOMEGA}.

\end{proof}    

\section{The Other half}
\mylabel{sec:other}
In this section we prove the transformations for the two other
components of \eqn{S-transformation}. Using similar methods
as described in Section 3, we arrive at two new identities, one
involving the mock theta function $\omega$ and the other involving
$f$. Proving these identities are equivalent to showing the two
remaining transformations of \eqn{S-transformation}. We observe
that \eqn{NEWOMEGA2} follows from \eqn{NEWOMEGA} by considering the
transformation $\tau \rightarrow 6\tau $ followed by $q\rightarrow
-q$. We also observe that the proof of this transformation implies
the proof of the \eqn{NEWF} since the transformations
\begin{align}
\frac{1}{\sqrt{-i\tau}}h_{0}\!\left(\frac{-1}{\tau}\right) &= h_{1}(\tau) 
\mylabel{eq:Stransh0},\\
\frac{1}{\sqrt{-i\tau}}h_{1}\!\left(\frac{-1}{\tau}\right) &= h_{0}(\tau) 
\mylabel{eq:Stransh1}
\end{align}
are clearly equivalent.

\subsection{Another identity involving the $\omega$ function}
\mylabel{subsec:anotherid}
From Watson's identity \eqn{fidWAT} and \cite[1.12, p.292]{JBT2022} we can 
write the mock theta function $f(q)$ in terms of Zwegers' $\mu$-function:
$$                  
q^{-1/24}f(q)
=\frac{\eta(3\tau)^4}{\eta(\tau)\eta(6\tau)^2}
  +4q^{-1/6}\mu\left(2\tau+1/2,\tau;3\tau \right).
$$              
Thus 
$$ 
h_{0}(\tau)=\frac{\eta(3\tau)^4}{\eta(\tau)\eta(6\tau)^2}
+4q^{-1/6}\mu\left(2\tau+1/2,\tau;3\tau \right)
 - 2i\sqrt{3}\int\limits_{-\overline{\tau}}^{i\infty}
  \frac{g_{1}(z)dz}{\sqrt{-i(z+\tau)}}.
$$               

To write $h_{0}$ in terms of Zwegers' $\mutwid$-function we need
$$ 
    g_{1}(z)=-\sum\limits_{n\in \mathbb{Z}}(n+1/6)e^{3\pi i (n+1/6)^2 z}
    =-\sum\limits_{n\in 1/6+\mathbb{Z}}  (n)e^{3\pi i (n)^2 z}
    =-g_{1/6,0}(3z).
$$               
Thus
\begin{align*}      
2i\sqrt{3}\int\limits_{-\overline{\tau}}^{i\infty}
    \frac{g_{1}(z)dz}{\sqrt{-i(z+\tau)}}
&=-2i\sqrt{3}\int\limits_{-\overline{\tau}}^{i\infty}
       \frac{g_{1/6,0}(3z)dz}{\sqrt{-i(z+\tau)}}
    =-2i\int\limits_{-\overline{3\tau}}^{i\infty}
      \frac{g_{1/6,0}(z)dz}{\sqrt{-i(z+3\tau)}}\\
&=2iq^{-1/6}R(-\tau+1/2;3\tau),
\end{align*}      
by substituting $a=-1/3$ and $b=-1/2$ in Theorem \thm{gabints}. 
Using Proposition \propo{Rellprops}(a) and (c), we have 
$$                     
h_{0}(\tau)=\frac{\eta(3\tau)^4}{\eta(\tau)\eta(6\tau)^2}
           +4q^{-1/6}\mutwid\left(2\tau+1/2,\tau;3\tau \right)
$$          
Using \eqn{Etatrans} and Proposition \propo{mutwidprops}(b) we find that 
\begin{align*}
    h_{0}\left(\frac{-1}{\tau}\right)
&=\frac{2\sqrt{-i\tau}\eta(\tau/3)^4}{3\eta(\tau)\eta(\tau/6)^2}+4e^{\frac{\pi i}{3\tau}}\mutwid\left(-2/\tau+1/2,-1/\tau;-3/\tau \right),\\
    &=\frac{2\sqrt{-i\tau}\eta(\tau/3)^4}{3\eta(\tau)\eta(\tau/6)^2}+4e^{\frac{\pi i}{3\tau}}\mutwid\left(\frac{(\tau-4)/6}{\tau/3},\frac{-1/3}{\tau/3};\frac{-1}{\tau/3} \right),\\
    &= \frac{2\sqrt{-i\tau}\eta(\tau/3)^4}{3\eta(\tau)\eta(\tau/6)^2} - 4q^{-1/24}e^{\pi i/3}\sqrt{\frac{-i\tau}{3}}\mutwid\left( (\tau-4)/6,-1/3; \tau/3\right).
\end{align*}
Thus,
\begin{align*}
&\frac{1}{\sqrt{-i\tau}}h_{0}\left(\frac{-1}{\tau}\right)
=\frac{2\eta(\tau/3)^4}{3\eta(\tau)\eta(\tau/6)^2} - \frac{4}{\sqrt{3}}q^{-1/24}e^{\pi i/3}\mutwid\left( (\tau-4)/6,-1/3; \tau/3\right)\\
&=\frac{2\eta(\tau/3)^4}{3\eta(\tau)\eta(\tau/6)^2} - \frac{4}{\sqrt{3}}q^{-1/24}e^{\pi i/3}\mu\left( (\tau-4)/6,-1/3; \tau/3\right)-\frac{2i}{\sqrt{3}}q^{-1/24}e^{\pi i/3}R(\tau/6-1/3;\tau/3).
\end{align*}
Substituting  $\tau \rightarrow \tau/3$ and $b=-1/3$ in Lemma \lem{Rext}   
we rewrite $R(\tau/6-1/3;\tau/3)$ in terms of an Eichler integral. 
Using Proposition \propo{Rellprops}(c) we have
\begin{align*}        
R(\tau/6-1/3) &= R(-\tau/6 + 1/2) 
=q^{1/24}e^{-\pi i/3}\left(1- i\int\limits_{-\overline{\tau/3}}^{i\infty}\frac{g_{0,1/6}(z)dz}{\sqrt{-i(z+\tau/3)}}\right)\\
&=q^{1/24}e^{-\pi i/3}\left(1- \frac{i}{\sqrt{3}}\int\limits_{-\overline{\tau}}^{i\infty}\frac{g_{0,1/6}(z/3)dz}{\sqrt{-i(z+\tau)}}\right).
\end{align*} 
We also make the following observation that
\begin{align*}
      g_{0}(z)=\frac{g_{1}(-1/z)}{-(-iz)^{3/2}}
      =\frac{-g_{1/6,0}(-3/z)}{-(-iz)^{3/2}}
      =\frac{-i(-iz/3)^{3/2}g_{0,-1/6}(z/3)}{-(-iz)^{3/2}}
      =\frac{-ig_{0,1/6}(z/3)}{3\sqrt{3}}.
\end{align*}
Thus,
\begin{align*}
    \frac{1}{\sqrt{-i\tau}} h_0\left(\frac{-1}{\tau}\right)
    &= \frac{2\eta(\tau/3)^4}{3\eta(\tau)\eta(\tau/6)^2} 
    - \frac{4}{\sqrt{3}} q^{-1/24} e^{\pi i/3} \mu\left(\frac{\tau - 4}{6}, -\frac{1}{3}; \frac{\tau}{3}\right)-\frac{2i}{\sqrt{3}} \\
    &\quad - 2i\sqrt{3} \int\limits_{-\overline{\tau}}^{i\infty} \frac{g_0(z)}{\sqrt{-i(z+\tau)}} \, dz.
\end{align*}
Since
\begin{equation}
h_1(\tau) = 2 q^{1/3} \omega(\sqrt{q}) - 
       2i\sqrt{3} \int_{-\overline{\tau}}^{i\infty} \frac{g_0(z)}{\sqrt{-i(z+\tau)}}dz,
\mylabel{eq:h1def}
\end{equation}
we see that the transformation \eqn{Stransh0} is equivalent to proving 
this new identity involving the mock theta function $\omega(q)$:
\begin{equation}
2q^{1/3} \omega(\sqrt{q})
    = \frac{2\eta(\tau/3)^4}{3\eta(\tau)\eta(\tau/6)^2} 
    - \frac{4}{\sqrt{3}} q^{-1/24} e^{\pi i/3}\mu\left(\frac{\tau - 4}{6}, -\frac{1}{3}; \frac{\tau}{3}\right)
    - \frac{2i}{\sqrt{3}}.
\mylabel{eq:newomid}
\end{equation}
By replacing $\tau)$ by $6\tau$ this identity is equivalent to 
\eqn{NEWOMEGA2}. 
It can be shown that \eqn{NEWOMEGA2} is equivalent to \eqn{NEWOMEGA} by
replacing $q$ by $-q$. Thus the transformation \eqn{Stransh0} holds.

\subsection{An identity for $f(q)$}
\mylabel{subsec:idfq}
From equations \eqn{h1def} and \eqn{newomid} we have
$$
h_{1}(\tau) = 2 \frac{\eta(3\tau)^{4}}{\eta(\tau)\eta(3\tau/2)^{2}} 
    - 4i q^{-1/24} \mu\left( \tfrac{3\tau}{2}, \tau; 3\tau \right) 
     - 2i \sqrt{3} \int\limits_{-\overline{\tau}}^{i\infty} 
    \frac{g_{0}(z) \, dz}{\sqrt{-i(z + \tau)}}.
$$                  
We note that
\begin{align*}
g_{0}(z)&=\sum\limits_{n\in \mathbb{Z}} (-1)^n (n+1/3)e^{3\pi i (n+1/3)^2 z}
        =\sum\limits_{n\in 1/3 + \mathbb{Z}} e^{\pi i(n-1/3)}\, n\, e^{3\pi i n^2 z} 
        =e^{-\pi i/3}g_{1/3,1/2}(3z).
\end{align*}
Thus
\begin{align*}     
       2i\sqrt{3}\int\limits_{-\overline{\tau}}^{i\infty}\frac{g_{0}(z)dz}{\sqrt{-i(z+\tau)}}
       &=2i\sqrt{3}e^{-\pi i/3}\int\limits_{-\overline{\tau}}^{i\infty}\frac{g_{1/3,1/2}(3z)dz}{\sqrt{-i(z+\tau)}}=2ie^{-\pi i/3}\int\limits_{-\overline{3\tau}}^{i\infty}\frac{g_{1/3,1/2}(z)dz}{\sqrt{-i(z+3\tau)}}\\
&= -2 q^{-1/24} R(-\tau/2;3\tau),
   \end{align*}
using Theorem \thm{gabints} with $a=-1/6$ and $b=0$.
Thus
$$                    
h_{1}(\tau)=2\frac{\eta(3\tau)^{4}}{\eta(\tau)\eta(3\tau/2)^{2}}-4iq^{-1/24}\mu(3\tau/2,\tau;3\tau)-4iq^{-1/24}\mutwid(3\tau/2,\tau;3\tau),
$$                   
and
\begin{align*}
 h_{1}\left(\frac{-1}{\tau}\right)
&=\frac{\sqrt{-i\tau}\eta(\tau/3)^{4}}{3\eta(\tau)\eta(2\tau/3)^{2}}
   -4ie^{2\pi i/24\tau}\mutwid(-3/2\tau,-1/\tau;-3/\tau)\\
&=\frac{\sqrt{-i\tau}\eta(\tau/3)^{4}}{3\eta(\tau)\eta(2\tau/3)^{2}}
   - 4ie^{2\pi i/24\tau}\mutwid\left(\frac{-1/2}{\tau/3},\frac{-1/3}{\tau/3};\frac{-1}{\tau/3}\right),\\
    &=\frac{\sqrt{-i\tau}\eta(\tau/3)^{4}}{3\eta(\tau)\eta(2\tau/3)^{2}}+4i\sqrt{-i\tau/3}\left(\mu(-1/2,-1/3;\tau/3)+\frac{i}{2}R(-1/6;\tau/3) \right),
\end{align*}
$$
\frac{1}{\sqrt{-i\tau}}h_{1}\left(\frac{-1}{\tau}\right)
= \frac{\eta(\tau/3)^{4}}{3\eta(\tau)\eta(2\tau/3)^{2}}
  +\frac{4i}{\sqrt{3}}\mu(-1/2,-1/3;\tau/3)
   -\frac{2}{\sqrt{3}}R(-1/6;\tau/3).
$$                   
Now from Theorem \thm{gabints} 
$$
R(-1/6;\tau/3)
=-\frac{1}{\sqrt{3}}\int\limits_{-\overline{\tau}}^{i\infty}
  \frac{g_{1/2,2/3}(z/3)dz}{\sqrt{-i(z+\tau)}}.
$$ 
We also note that                    
$$                     
      g_{1}(z)=\frac{g_{0}(-1/z)}{-(-iz)^{3/2}}
      =\frac{e^{-\pi i/3}g_{1/3,1/2(-3/z)}}{-(-iz)^{3/2}}
=\frac{e^{-\pi i/3}ie^{\pi i/3}(-iz/3)^{3/2}g_{1/2,-1/3(z/3)}}{-(-iz)^{3/2}}
    =\frac{ig_{1/2,2/3}(z/3)}{3\sqrt{3}}
$$                       
Thus $g_{1/2,2/3}(z/3)=-i3\sqrt{3}g_{1}(z)$.
Therefore showing \eqn{Stransh1} is equivalent to the following:
$$              
\frac{\eta(\tau/3)^{4}}{3\eta(\tau)\eta(2\tau/3)^{2}}
 + \frac{4i}{\sqrt{3}}\mu(-1/2,-1/3;\tau/3)
 -  2i\sqrt{3}\int\limits_{-\overline{\tau}}^{i\infty}
     \frac{g_{1}(z)dz}{\sqrt{-i(z+\tau)}}
 = q^{\frac{-1}{24}}f(q)
  - 2i\sqrt{3}\int\limits_{-\overline{\tau}}^{i\infty}
     \frac{g_{1}(z)dz}{\sqrt{-i(z+\tau)}};
$$                   
i.e.  
$$                    
q^{\frac{-1}{24}}f(q)
 = \frac{\eta(\tau/3)^{4}}{3\eta(\tau)\eta(2\tau/3)^{2}}
   + \frac{4i}{\sqrt{3}}\mu(-1/2,-1/3;\tau/3).
$$               
By replacing $\tau$ by $3\tau$ this identity is equivalent to 
\eqn{NEWF}. This proves \eqn{NEWF} since \eqn{Stransh0} and
\eqn{Stransh1} hold.

\section{Conclusion}
In this paper we prove Zwegers' transformations for the mock theta
functions $f$ and $\omega$ without using Watson's transformation
formulas\cite{Wa1936} and instead use Appell-Lerch sums from Zwegers'
thesis \cite{Zw2002}. In the process we discover new identities
for mock theta functions \eqn{NEWOMEGA},\eqn{NEWOMEGA2} and
\eqn{NEWF}.We prove \eqn{NEWOMEGA} by calculating a $3$-dissection
of each term on the right. Identities \eqn{NEWOMEGA2} and
\eqn{NEWF} are connected to two identities in Ramanujan's Lost
Notebook in \eqn{rlnomega} and \eqn{rlnF} below. We pose the
problem of showing if \eqn{rlnomega} and \eqn{rlnF} can be proved
from \eqn{NEWOMEGA2} and \eqn{NEWF} or vice versa. 

Define
$$       
\psi(q) := \sum_{n=0}^{\infty} q^{n(n+1)/2}
= \frac{(q^2;q^2)_\infty}{(q;q^2)_\infty},\quad 
\phi(q) := \sum_{n=-\infty}^{\infty}(-1)^{n} q^{n^{2}}
= \frac{(q;q)_\infty}{(-q;q)_\infty}.
$$        
Then
\begin{align}    
\sum_{n=0}^{\infty} 
\frac{q^{6n^2}}{(q;q^6)_{n+1}(q^5;q^6)_n}
&= \frac{1}{2}\Biggl(
1 + q^2 \omega_3(q^3) + \frac{\psi^2(q)}{(q^6;q^6)_\infty}
\Biggr),
\mylabel{eq:rlnomega}
\\
\sum_{n=0}^{\infty} 
\frac{(-1)^n (q;q)_{2n} q^{n^2}}{(q^6;q^6)_n}
&= \frac{3}{4} f_3(q^3) 
   + \frac{1}{4}\,\frac{\varphi^2(-q)}{(q^3;q^3)_\infty}.
\mylabel{eq:rlnF}
\end{align}   
A natural question is to derive modular transformations for
the completions of all of Ramanujan's mock theta functions
using Appell-Lerch sums.  Klein and Kupka \cite{KleinKupka2021}
used results from Gordon–McIntosh \cite[pp.109-130]{GM2012} to
work out the transformations for the completions of mock theta
functions of order $2$, and the remaining $3$rd, $6$th, and $8$th
order cases. In a forthcoming paper we consider this question
and give similar identities for mock theta functions analogous to
\eqn{NEWOMEGA}-\eqn{NEWF}.

\textbf{Acknowledgments}.
The authors would like to thank George Andrews, Ole Warnaar and Jonathan 
Bradley-Thrush for their helpful comments and suggestions on some 
of the identities mentioned in the paper.

\bibliographystyle{amsplain}


\end{document}